\documentclass[letterpaper, 10 pt, conference]{ieeeconf} 

\IEEEoverridecommandlockouts                              
\usepackage{cite}
\usepackage{amsmath,amssymb,amsfonts,amsbsy}
\usepackage{algorithmic}
\usepackage{graphicx}
\usepackage{textcomp}
\usepackage{xcolor}
\usepackage{array}
\usepackage{float}
\usepackage{multirow}
\usepackage{times}
\usepackage{comment}
\usepackage{url,theorem}
\usepackage{comment}
\usepackage{svg}

{\theorembodyfont{\slshape}\newtheorem{theorem}{Theorem}[section]}
{\theorembodyfont{\slshape}\newtheorem{proposition}[theorem]{Proposition}}
{\theorembodyfont{\slshape}\newtheorem{lemma}[theorem]{Lemma}}
{\theorembodyfont{\slshape}}
{\theorembodyfont{\slshape}\newtheorem{corollary}[theorem]{Corollary}}
{\theorembodyfont{\upshape}\newtheorem{definition}[theorem]{Definition}}
{\theorembodyfont{\upshape}}
{\theorembodyfont{\upshape}\newtheorem{remark}[theorem]{Remark}}
{\theorembodyfont{\upshape}}

\newcommand{\setR}{\mathbb{R}}

\newcommand{\graphA}{\mathcal{A}}

\newcommand{\graphC}{\mathcal{C}}

\newcommand{\graphK}{\mathcal{K}}
\newcommand{\graphR}{\mathcal{R}}
\newcommand{\graphS}{\mathcal{S}}
\newcommand{\graphT}{\mathcal{T}}

\newcommand{\graphW}{\mathcal{W}}
\newcommand{\asset}{\mathcal{A}_{\graphW\graphC}}

\newcommand{\Lagr}{\mathcal{L}}
\newcommand{\tangcone}{\graphT_{\asset}}
\newcommand{\nonnegR}{\mathbb{R}_{\geq0}}
\newcommand{\posR}{\mathbb{R}_{>0}}
\newcommand{\setRn}{\mathbb{R}^{n}}
\newcommand{\setRm}{\mathbb{R}^{m}}

\newcommand{\Sontaggammul}[3]{
\begin{cases}-\frac{a(#2)+\gamma\sqrt{a(#2)^2+\|b(#3)^\top\|^4_2}}{\|b(#3)^\top\|^2_2}b(#3)^\top   &\text{if } \|b(#3)^\top\|^2_2>0\\
   0  &\text{if } \|b(#3)^\top\|^2_2 =0\\ 
  \end{cases}}

\title{\LARGE \bf
	On unifying control barrier and Lyapunov functions using QP and Sontag's formula with an application to tumor dynamics 
}

\author{Jarne J.H. van Gemert, Mircea Lazar, Siep Weiland
	\thanks{All authors are with the Department of Electrical Engineering, Eindhoven University of Technology, The Netherlands, E-mail of corresponding author: {\tt\small j.j.h.v.gemert@tue.nl}.}
}

\begin{document}
\maketitle
\thispagestyle{plain}
\pagestyle{plain}
\begin{abstract}
A common tool in system theory for formulating control laws that achieve local asymptotic stability are Control Lyapunov functions (CLFs), while Control Barrier functions (CBFs) are typically employed to enforce safety constraints. Combining these two types of functions is of interest, because it leads to stabilizing controllers with safety guarantees. A common approach to merge CLFs and CBFs is to solve an optimization problem where both CLF and CBF inequalities are imposed as constraints. In this paper, we show via an example from the literature that this approach can lead to undesirable behavior (i.e., slow convergence and oscillating inputs). Then, we propose a novel cost function that penalizes the deviation from Sontag's formula by using a state-dependent weighting matrix. We show that by minimizing the developed cost function subject to a CBF constraint, local asymptotic stability is obtained with an explicit domain of attraction, without using a CLF constraint. To deal with vanishing properties of the weight matrix as the state approaches the equilibrium, we introduce a hybrid continuous control law that recovers Sontag's formula locally. The effectiveness of the developed hybrid stabilizing control law based on CLFs and CBFs is illustrated in stabilization of a 3D tumor model, subject to physiological constraints (i.e., all states must be positive), which yields useful insights into optimal cancer treatment design.
\end{abstract}

\section{Introduction}\label{Sec: Introduction}
Control Lyapunov functions (CLFs), introduced in \cite{artstein1983stabilization},  have been commonly used to develop stabilizing control laws for nonlinear systems. Based on the CLF theorem,  analytically derived ``universal'' control laws can be developed such as Sontag's  ``universal'' control law introduced in \cite{sontag1989Universal}. Another way to use CLFs is by constructing a point-wise minimum norm (PMN) controller based on a quadratic programming (QP) problem, such as the one introduced in \cite{freeman1996robust}. 
Closely related to CLFs are Control Barrier functions (CBFs). The latter are safety-related functions used to enforce safety constraints for controlled systems. CBFs define a set of safe states for the system and act as a barrier to prevent the system from entering undesirable or unsafe regions. Safety of trajectories can be enforced by solving a QP problem which uses the CBF condition as an explicit constraint, as introduced in \cite{ames2014control,ames2017}.  
By also adding the CLF decrease condition to such a QP problem, one can combine CLFs and CBFs and achieve safety together with local asymptotic stability (LAS), which is desirable for most real-life systems. For example, for biological models, such as tumor models, physiological constraints require that the state trajectories must remain positive at all times, \cite{doban2015evolutionary}. 

In the original CBF-QP problem formulation introduced in \cite{ames2014control,ames2017,ames2019control} a cost function was minimized that penalizes the deviation from the computed control action and a desired nominal controller, e.g., such as Sontag's unconstrained stabilizing control law. However, this approach does not necessarily guarantee asymptotic stability. Subsequently, the CLF-CBF-QP problem was introduced, see e.g., \cite{ames2014control}, which utilizes an explicit CLF decrease constraint in order to  guarantee stability and minimizes the norm of the computed control action. This method required the introduction of an additional parameter, i.e., a slack variable, $\delta$, to ensure feasibility of the combined CLF and CBF constraints. The introduction of $\delta$, presented another challenge; the states are not guaranteed to converge towards the desired equilibrium. To address this challenge, different methodologies were proposed in recent studies. 
In \cite{jankovic2018robust}, the original CLF-CBF-QP problem was modified such that the CLF condition is satisfied if the CBF constraint is inactive, which yields LAS 
without explicitly constructing a domain of attraction (DOA). However, the utilized cost function therein still makes a trade-off between minimizing the norm of the computed control action and the weighted norm of $\delta$. Similarly, in \cite{tan2021undesired}, it was proposed to combine the CLF-CBF-QP formulation with a cost function that penalizes the deviation of the computed control action from a nominal, known stabilizing controller. In this approach LAS can be guaranteed if the weight of $\delta$ is properly tuned, but in principle, this optimal tuning should be done for every initial state, which is impractical. 

In this paper we show for an example from \cite{tan2021undesired} that although both methods above provide safety and LAS, the tuning of the weight on $\delta$ can result in undesired behaviour, such as getting stuck into an extremely slow convergence behavior or generating highly oscillating inputs. Motivated by this, in this paper we derive an alternative solution to unify CLFs and CBFs based on the original \emph{CBF-QP} formulation \cite{ames2014control}. As proposed therein we only use a CBF constraint, we minimize a cost function that penalizes the deviation of the computed control action from Sontag's unconstrained stabilizing control law (i.e., a nominal controller) and we introduce a novel, state-dependent weighting matrix which implicitly guarantees that the CLF condition holds, whenever this is feasible. In this way, we remove the need to explicitly incorporate a relaxed CLF constraint and as such, we remove the need to use an additional variable $\delta$ and the problems coming with it. However, the proposed approach suffers from a different challenge, i.e., the state dependent weighting matrix vanishes when the state approaches the equilibrium. To solve this issue, we introduce a hybrid stabilization formula which switches from solving the developed CBF-QP problem to Sontag's formula when the CBF constraint becomes inactive. We prove that the designed hybrid Sontag-CBF-QP solution (or S-CBF-QP for short) yields a continuous control law and an enlarged domain of attraction compared to the CLF-CBF-QP solutions in the literature. 

To showcase the benefits of the developed S-CBF-QP stabilizing solution, we consider the challenging problem of stabilizing a 3D nonlinear tumor model \cite{doban2015evolutionary}. This model describes the tumor, resting and hunting immune cells as competing species and can represent several therapeutic strategies by altering key parameters in the state equations. For this paper, the control input is chosen such that it influences a parameter that can be directly influenced by immunotherapy, i.e. representing the effect of immunotherapy on the tumor and healthy cells. Given this model represents a biological system, it is subject to physiological constraints, i.e. all states are constrained to be positive. While the original autonomous tumor model is guaranteed to have positive states, the controlled tumor model is not. When applying an unconstrained controller to the tumor model, such as Sontag's unconstrained stabilizing control law, the controlled states can become negative and result in physiologically unrealistic trajectories. Therefore, for this tumor model, we design a CBF that restricts trajectories to the positive orthant. Then the hybrid S-CBF-QP method can be applied to safely stabilize the tumor dynamics to a desired healthy equilibrium.

\section{Preliminaries}\label{Sec: preliminaries}
Let $\setR$, $\nonnegR$ and $\posR$ denote the field of real numbers, the set of non-negative reals and the set of positive reals, respectively. 
For a vector $x\in \setRn$, $x_i$ denotes the $i$-th element of $x$, $\|x\|_2=\sqrt{x_1^2+...+x_n^2}$ denotes the two-norm of $x$ and $\|x\|_2^2=x_1^2+...+x_n^2=x^\top x$ the squared two-norm of $x$.  For any suitable function $f\in\setR^{n\times n}$ we introduce the notation $\|x\|_{f(x)}^2$ to denote $=x^\top f(x)x$. For a matrix $A\in \setR^{n\times n}$, $A^\top$ denotes its transpose and $A\succ 0$ indicates that ``$A$ is positive definite'', i.e. for all $x\in \setRn\setminus\{0\}$ it holds that $x^\top Ax>0$ and $A=A^\top$.  
A set $\graphS\subseteq\setRn$ is called a proper set if it has a non-empty interior. Moreover, it is called $\xi$-proper for some $\xi\in\setR^n$ if it contains $\xi$ in its interior. Let $\partial\graphS$ denote the boundary of $\graphS$ and $int(\graphS)$ the interior of $\graphS$. The distance from a point $x\not\in\graphS$ to the set $\graphS$ is given by $\|x\|_{\graphS}:=\text{inf}_{y\in\graphS}\|x-y\|$. For a closed set $\graphS$, 
the Boulingand's tangent cone is defined by 
\begin{equation}
    \graphT_\graphS(x)=\biggl\{y:\liminf_{\tau\rightarrow 0}\frac{\|x+\tau y\|_{\graphS}}{\tau}=0\biggr\},
\end{equation}
following Definition 4.6 of \cite{blanchini2008set}.
A function $\alpha\;:\;\nonnegR\rightarrow\nonnegR$ belongs to class $\graphK$ if it is continuous, strictly
increasing and $\alpha(0)=0$. Additionally, if  $\text{lim}_{s\rightarrow\infty}\alpha(s)=\infty$, then $\alpha$ belongs to class $\graphK_\infty$. 
A function $\alpha\;:\;\setR\rightarrow\setR$ belongs to \emph{extended class} $\graphK$ \cite{ames2016control} if it is strictly increasing, $\alpha(0)=0$ and $\text{lim}_{s\rightarrow a}\alpha(s)=\infty$, $\text{lim}_{s\rightarrow\- -b}\alpha(s)=-\infty$ for some $a,b\in\posR$.

Consider the continuous-time, time-invariant, nonlinear system 
\begin{equation}\label{eq:2}
    \Dot{x}(t)=\Bar{f}(x(t),u(t)), \quad t\in \nonnegR,
\end{equation}
where $x(t)\in \setRn$ and $u(t)\in \setRm$ are the continuous-time state and input trajectories and where $\bar f:\setRn\times\setRm\rightarrow\setRn$ is a globally Lipschitz continuous function. For this paper, only the class of control systems which are affine with respect to the input are considered, i.e., 
\begin{equation}\label{eq:2.1} 
    \Dot{x}(t)=f(x(t))+g(x(t))u(t), \quad t\in \nonnegR,
\end{equation}
where $x(t)\in \setRn$ and $u(t)\in \setRm$ are the continuous-time state and input trajectories and where $f:\setRn\rightarrow\setRn$ and
$g:\setRn\rightarrow \setRn \times \setR^m$ are globally Lipschitz continuous functions. 
The solution of system \eqref{eq:2.1} with initial condition $x(0)=x_0$ and input $u$ at time $t$ is denoted by $x(t,u,x_0)$.
For brevity, any such solution is denoted by $x(t)$ or $x$ if $x_0$, time $t$ and $u$ are understood. A pair $(x_e,u_e)$, denotes an equilibrium point of system \eqref{eq:2.1}, if $f(x_e)+g(x_e)u_e=0$. 

\begin{definition}[\hspace{-0.01cm}\cite{Khalil:1173048}]\label{Def: eq}
\emph{(i)} An equilibrium pair $(x_e,u_e)$ of system \eqref{eq:2.1}, with known control input $u(t)$, is \emph{Lyapunov stable} if for any $\epsilon>0$ there exists $\delta>0$ such that for all corresponding state trajectories of \eqref{eq:2.1} it holds that $\|x(0)-x_e\|\leq \delta\Rightarrow\|x(t)-x_e\|\leq\epsilon$ for all $t\geq0$. \emph{(ii)} We call an equilibrium pair $(x_e,u_e)$ of system \eqref{eq:2.1}, with control input $u(t)$, \emph{attractive in} $\graphS\subseteq\setRn$ if for each $x(0)\in\graphS$ it holds that all corresponding state trajectories of \eqref{eq:2.1} satisfy $\lim_{t\rightarrow\infty}\|x(t)-x_e\|=0$. \emph{(iii)} We call an equilibrium pair $(x_e,u_e)$ of system \eqref{eq:2.1}, with control input $u(t)$, asymptotically stable (AS) in $\graphS$ if it is Lyapunov stable and attractive in $\graphS$.
\end{definition}

\begin{definition}[\hspace{-0.01cm}\cite{Khalil:1173048}]\label{Def: control invariant}
    A set $\graphS\subseteq\setRn$ is called a controlled invariant set for system \eqref{eq:2.1}, if there exists a state-feedback law $u(t):=\kappa(x(t))$, such that for each $x(0)\in \graphS$, $x(t)\in \graphS$ for all $t\geq 0$.
\end{definition}

Next, we define a safe set as follows:
\begin{align}\label{eq: safeset}
    \graphC &= \{x\in\setRn:h(x)\geq0\},
    \end{align}
    with boundary and interior defined as
    \begin{align}
    \partial\graphC &=\{x\in\setRn:h(x)=0\},\\
    \textit{int}(\graphC) &= \{x\in\setRn:h(x)>0\},
\end{align}
where $h:\setRn\rightarrow\setR$ is a twice differentiable function referred to as a barrier function. 
For this paper we consider \emph{concave} functions for $h$, which result in \emph{convex} safe sets. 

\begin{definition}[\hspace{-0.01cm}\cite{ames2019control}]\label{Def: CBF}
    Let a safe set $\graphC$ be defined by \eqref{eq: safeset} with respect to a barrier function $h$. Then $h$ is called a Control Barrier function (CBF) for system \eqref{eq:2.1}, if there exists an extended class $\graphK$ function $\alpha_h$ and a input $u\in\setRm$ such that
\begin{equation}
    \begin{aligned}
        \nabla h(x)^\top(f(x)+g(x)u)\geq -\alpha_h(h(x)),\quad \forall x\in\graphC.
    \end{aligned}
\end{equation} 
\end{definition}

\begin{proposition}[\hspace{-0.01cm}\cite{ames2014control}]
    If $h$ is a Control Barrier function for system \eqref{eq:2.1}, then $\graphC$ is a controlled invariant set.
\end{proposition}
Next, we provide a control Lyapunov function definition.

\begin{definition}[\hspace{-0.01cm}\cite{Khalil:1173048}]\label{Def: CLF}
Let $W:\setRn\rightarrow\nonnegR$ be a continuously differentiable function that satisfies 
\begin{align}\label{eq: cond1CLF}
&\alpha_1(\|x-x_e\|)\leq W(x)\leq\alpha_2(\|x-x_e\|), \quad \forall x\in\graphS 
\end{align}
where $\graphS\subseteq\setRn$ for some functions $\alpha_1,\alpha_2\in\graphK_\infty$. Then, if there exists a state-feedback control law $u:=\kappa(x)$ with $\kappa(x_e)=u_e$ such that
\begin{align}
&\nabla W(x)^\top (f(x)+g(x)\kappa(x))< 0, \quad \forall x\in\graphS\setminus\{x_e\}, 
\end{align}
the function $W$ is called a Control Lyapunov function (CLF) for system \eqref{eq:2.1} for all $x\in\graphS$.
\end{definition}
\vspace{-1mm}

Note that the above definition corresponds to a regional CLF, i.e., a CLF in $\graphS$. $W$ is a global CLF if $\graphS=\mathbb{R}^n$. Following standard Lyapunov arguments, the set $\graphS$ and any sub-level set of $W$ contained in $\graphS$ is rendered invariant by the feedback law $u=\kappa(x)$ for system \eqref{eq:2.1} \cite{Khalil:1173048}. For this paper, only \emph{convex} functions are considered for $W$, which result in \emph{convex} level sets of $W$. 

\vspace{-3mm}
\begin{definition}[\hspace{-0.01cm}\cite{krstic1998stabilization}]
    The CLF $W$ is said to have the small control property if there exists a continuous state-feedback control law $u=\kappa(x)$ where $\kappa(x_e)=u_e$, that satisfies
    \begin{align}\label{eq:sconprop}
        \nabla W(x)^\top (f(x)+g(x)\kappa(x))< 0,
    \end{align}
for all $x$ in the neighborhood of the equilibrium point $x_e$ with $x\neq x_e$.
\end{definition}

A ``universal'' stabilizing continuous control law based on CLF was introduced by Sontag in \cite{sontag1989Universal}. This ``universal'' control law is explicitly designed to globally stabilize nonlinear systems which are affine with respect to the control input. 

Next, we present a slightly adjusted version of Sontag's formula, valid for systems with multiple inputs, non-zero equilibria and which allows for tuning the rate of decrease of the Lyapunov function.
Let $(x_e,u_e)$ be an equilibrium pair of system \eqref{eq:2.1}, let $W$ be a CLF for system \eqref{eq:2.1}. Then define Sontag's formula as:
\begin{equation}
\label{eq:u_son}
u_{son}(x,u_e) = \kappa(x,u_e)+u_e,
\end{equation}
where
\begin{align*}
&\kappa(x,u_e):=\\&\Sontaggammul{\kappa}{x,u_e}{x}, 
\end{align*}
and where
$a(x,u_e):=\nabla W(x)^\top (f(x)+g(x)u_e)$, $b(x):=\nabla W(x)^\top g(x)$ and $\gamma\in \posR$ influences the convergence rate.
 
\begin{lemma}\label{lemma:Sontag} Consider system \eqref{eq:2.1} with input $u=u_{son}(x,u_e)$ as given in \eqref{eq:u_son}. Assume there exists a CLF $W$ in a $x_e$-proper set $\graphS\subseteq\setRn$ for system \eqref{eq:2.1}. Then the state feedback $u=u_{son}(x,u_e)$ is smooth and continuous at $x_e$ and, if $\graphS$ is rendered invariant by $u$, then the equilibrium point $(x_e,u_e)$ of system \eqref{eq:2.1} is asymptotically stable in $\graphS$. 
\end{lemma} 

\begin{proof}
Given $W$ is a CLF for system \eqref{eq:2.1} for equilibrium pair $(x_e,u_e)$, we can use the CLF condition \eqref{eq: cond1CLF}, to establish that
$\nabla W(x_e)=0$ and thus
\begin{align}
    a(x_e,u_e)=b(x_e)=\kappa(x_e,u_e)=0,
\end{align} 
which results in $u_{son}(x_e,u_e)=u_e$.
Let us prove that $u_{son}$ asymptotically stabilizes the equilibrium point $(x_e,u_e)$ of system \eqref{eq:2.1} for all $x\in \setRn$. For simplicity, $a(x,u_e)$ is denoted by $a$, $b(x)$ is denoted by $b$ and $u_{son}(x,u_e)=\kappa(x,u_e)+u_e$ is denoted by $u_{son}=\kappa+u_e$. Then we have that
\begin{equation*}
    \begin{aligned}
        \nabla W(x)^\top&(f(x)+g(x)u_{son})\\
        =&\nabla W(x)^\top(f(x)+g(x)(\kappa+u_e))\\
        =&\nabla W(x)^\top(f(x)+g(x)u_e)+\nabla W(x)^\top g(x)\kappa(x)\\
        =&a+b\kappa\\
        =&a+bb^\top\frac{-a-\gamma\sqrt{a^2+\|b^\top \|_2^4}}{\|b^\top \|^2_2}\\
        =&a+\|b^\top \|^2_2\frac{-a-\gamma\sqrt{a^2+\|b^\top \|_2^4}}{\|b^\top \|^2_2}\\
        =&-\gamma\sqrt{a^2+\|b^\top \|_2^4}<0, \quad \forall x\in\graphS \; \setminus \{x_e\}.
    \end{aligned}
\end{equation*}
Thus $u_{son}$ asymptotically stabilizes equilibrium point $(x_e,u_e)$ of system \eqref{eq:2.1} for all $x\in \graphS\subseteq\setRn$. The claim follows from standard Lyapunov arguments, see, for example \cite{Khalil:1173048}.

Additionally, the CLF $W$ satisfies the small control property as $u_{son}(x_e,u_e)=x_e$ and it satisfies \eqref{eq:sconprop}. Following Theorem 1 of \cite{sontag1989Universal}, we have that $u_{son}(x,u_e)$ is smooth on $\graphS\subseteq\setRn\setminus\{x_e\}$ and continuous at $x_e$.
\end{proof}

\subsection{CBF-QP}
Although Sontag's control law can globally stabilize equilibria, it cannot guarantee satisfaction of constraints on controlled state trajectories, which is a substantial limitation for many real-life systems. As a result, researchers explored the approach of combining CLFs with CBFs to obtain stability and safety guarantees. 

Within this context, 

two commonly used quadratic programming (QP) problems have been formulated in order to compute a pointwise input while satisfying the safety constraints defined by a CBF. Given a desired control input $u_{nom}(x)$, originally a CBF-QP was defined as follows in \cite{ames2014control}.\newline

\noindent\textbf{CBF-QP problem:} 
\begin{equation}\label{eq: basis cbf-qp}
    \begin{aligned}
        \min_{u} \quad & ||u-u_{nom}(x)||^2_2 \\
        \textrm{s.t.} \quad
        &\nabla h(x)^\top f(x)+\nabla h(x)^\top g(x)u \geq-\alpha_h(h(x)).
    \end{aligned}
\end{equation}
Although this standard CBF-QP problem guarantees the forward invariance of the safe set $\graphC$, it does not have any (local) asymptotically stability guarantees, \cite{cortez2019control}.   
Therefore, an alternative problem was introduced in \cite{ames2014control} as follows.\newline

\noindent\textbf{CLF-CBF-QP problem:} 
\begin{equation}\label{eq: basis CLF-CBF-QP}
    \begin{aligned}
        \min_{u,\delta} &\quad  ||u||^2_2 +p\delta^2\\
        \textrm{s.t.} \quad
        &\nabla W(x)^\top f(x)+\nabla W(x)^\top g(x)u\leq-\alpha_W(W(x))+\delta\\
        &\nabla h(x)^\top f(x)+\nabla h(x)^\top g(x)u \geq-\alpha_h(h(x)),
    \end{aligned}
\end{equation}
where $\alpha_\graphW(W(x))\in\graphK_\infty$, $p>0$ and $\delta$ is an additional optimization parameter to guarantee feasibility. 
This baseline CLF-CBF-QP problem was later modified/analyzed in \cite{jankovic2018robust,tan2021undesired} in order to guarantee local asymptotic stability (LAS).

\subsection{Motivating example and problem statement}\label{subsec: Example}
Next, we analyze the sensitivity of the CLF-CBF-QP problem defined in \cite{tan2021undesired}, to the weight $p$ of the parameter $\delta$, based on a linear system example taken from \cite{tan2021undesired}, i.e. 
consider system \eqref{eq:2.1} with 
\begin{align}\label{eq:ex system}
    f(x)=&-\begin{bmatrix}
        x_2\\
        x_1
    \end{bmatrix},\quad
    g(x)=\begin{bmatrix}
        0\\1
    \end{bmatrix},
\end{align}
with the desired equilibrium point given by $(x_e,u_e)=\left(\begin{pmatrix}0& 0\end{pmatrix}^\top,0\right)$. The nominal controller is taken as $u_{son}(x)$ as defined in \eqref{eq:u_son}. The technique of \cite{boyd1994linear} was used to compute a quadratic CLF $W(x)=x^\top Px$ where 
\begin{equation*}
    P=\begin{bmatrix}
        3.4142 & -2.4142\\
   -2.4142 & 2.4142
   \end{bmatrix}.
\end{equation*}
The CBF is given by $h(x)=-0.1x_1^2-0.15x_1x_2-0.1x_2^2$ \cite{tan2021undesired}, with the corresponding convex safe set depicted as the white area in Figure \ref{fig:motivexamp}.

\begin{figure}[H]
    \centering
    \includegraphics[width=0.98\linewidth]{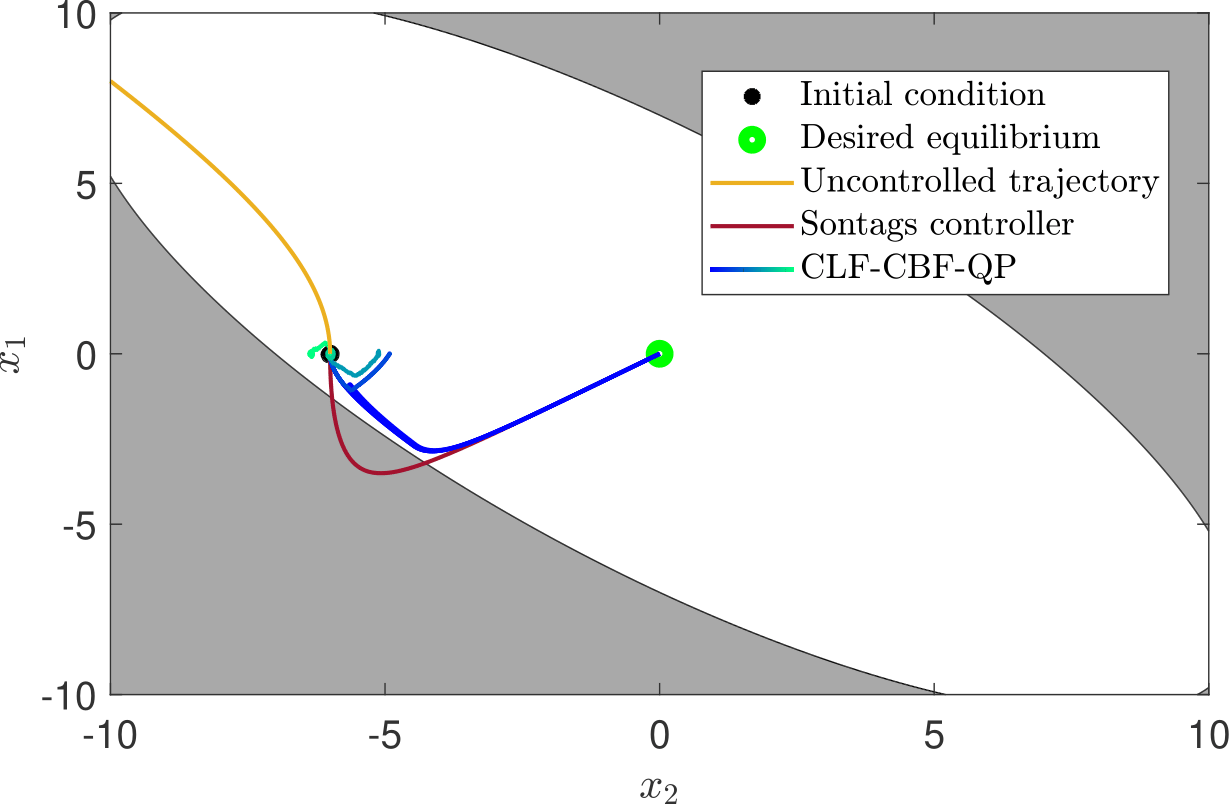}
\end{figure}
\vspace{-3mm}
\begin{figure}[H]
    \centering
    \includegraphics[width=1\linewidth]{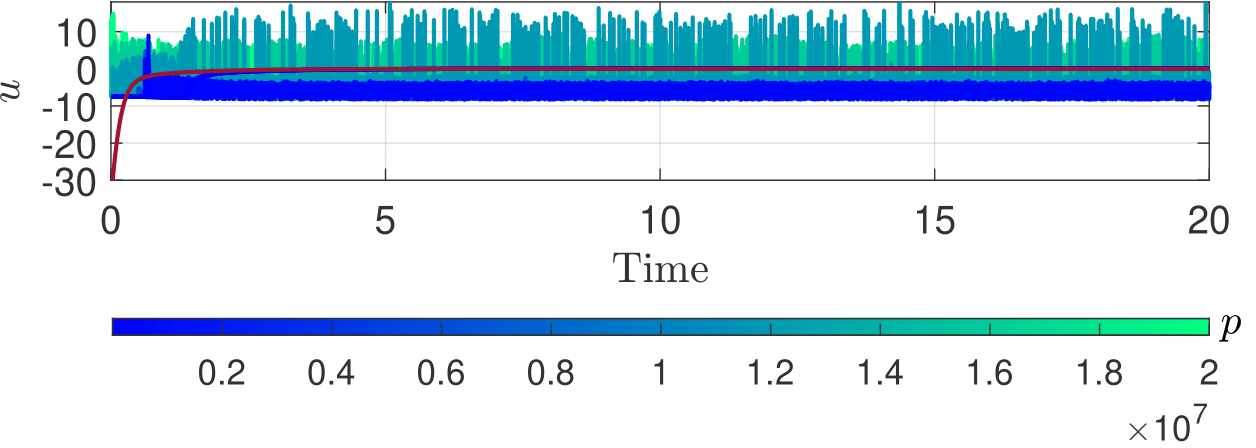}
    \caption{Example with the standard CLF-CBF-QP problem, for different $p$ with the values of $p$ defined in the color bar.}
\label{fig:motivexamp}
\end{figure}

As it can be seen in Figure \ref{fig:motivexamp}, for some initial condition and different values of $p$, the CLF-CBF-QP problem yields a highly oscillating input. Additionally, for higher values of $p$, trajectories get stuck into an extremely slow converging behavior. This is caused by the fact that the high penalty on $\delta$ makes the CLF constraint active, which conflicts with the CBF constraint. It is worth to point out that we obtained a similar behavior using the CLF-CBF-QP problem variations in both \cite{jankovic2018robust} and \cite{tan2021undesired}. 

It is also important to note that for CLF-CBF-QP problems, in current research, the domain of attraction (DOA), - that is, the region where the closed-loop equilibrium is asymptotically stable - has not been explicitly characterized; except for the trivial DOA estimate provided by the largest level set of the CLF contained inside the safe set, which can be conservative.

\emph{Problem statement: } Given the above analysis of CLF-CBF-QP problems, our objective is to develop a stabilizing, continuous control law based on the original CBF-QP formulation \cite{ames2014control}, which guarantees LAS without imposing the relaxed CLF constraint and therefore removing the need of the additional parameter $\delta$. Moreover, we would like to guarantee LAS for an explicitly defined DOA estimate, which is less conservative than the largest level set of the CLF contained inside the safe set.  

\section{Main results}
Consider the nonlinear control affine system \eqref{eq:2.1} with a CLF $W$ and a CBF $h$. As typically done, we assume that there exists an equilibrium pair $(x_e,u_e)$ for system \eqref{eq:2.1}, where the state variable $x_e$ belongs to interior of the safe set $\graphC$. 

Next we introduce an alternative CBF-QP formulation.\newline
\noindent\textbf{S-CBF-QP problem} 
\begin{equation}\label{eq: Son-QP}
    \begin{aligned}
        \min_{u} \quad & ||u-u_{son}(x)||^2_{g(x)^\top \nabla W(x)\nabla W(x)^\top g(x)} \\
        \textrm{s.t.} \quad &\nabla h(x)^\top f(x)+\nabla h(x)^\top g(x)u\geq -\alpha_h(h(x)),
    \end{aligned}
\end{equation} 
where the nominal controller is defined as the modified version of Sontag's ``universal'' control law defined in \eqref{eq:u_son}. For simplicity of exposition define 
\begin{align*}
    Q(x)&:=g(x)^\top \nabla W(x)\nabla W(x)^\top g(x),\\
    L_fh(x)&:=\nabla h(x)^\top f(x),\\
    L_gh(x)&:=\nabla h(x)^\top g(x).
\end{align*}
The optimal control input resulting from the optimization problem is given by
\begin{equation}\label{eq:armin Son-CBF-qp}
\begin{aligned}
    u^\ast (x)=\;\text{arg} \min_u \; &\|u-u_{son}(x)\|^2_{Q(x)}\\
    \textrm{s.t.} \quad &L_fh(x)+L_gh(x)u\geq -\alpha_h(h(x)).
\end{aligned}
\end{equation} 
Above, recall that \[\|u-u_{son}(x)\|^2_{Q(x)}=(u-u_{son}(x))^\top Q(x)(u-u_{son}(x)).\]
A property of the CLF $W$, is that its gradient satisfies $\nabla W(x_e)=0$. Consequently, the cost function of the S-CBF-QP problem is not well defined as $x$ approaches $x_e$, i.e., it vanishes asymptotically. Therefore, we introduce a hybrid stabilization formula which switches from solving the developed CBF-QP problem to Sontag's formula  when the CBF constraint becomes inactive. To that end let us define two regions 
\begin{align}
    \graphR_1 &:=\{x\in\graphC\;|\; L_fh(x)+\alpha_h(h(x))\\ &\hspace{3.5cm}+L_gh(x)u_{son}(x)\geq 0\},\nonumber \\
    \graphR_2 &:=\{x\in\graphC\;|\; L_fh(x)+\alpha_h(h(x))\\ &\hspace{3.5cm}+L_gh(x)u_{son}(x) \leq 0\}.\nonumber
\end{align}
We assume that $\graphR_1$ is a $x_e$-proper set. The hybrid S-CBF-QP control law is given as follows:
\begin{equation}\label{Eq: Hybrid control}
\begin{aligned}
       &\phi (x):=
       \begin{cases}
        u_{son}(x) & \text{if } x\in\graphR_1\\
        u^\ast(x) & \text{if }x\in\graphR_2
       \end{cases}. 
\end{aligned} 
\end{equation}

Before presenting the main result, we first present some properties of the S-CBF-QP problem in \eqref{eq:armin Son-CBF-qp} and of the hybrid S-CBF-QP control law in \eqref{Eq: Hybrid control}.

\begin{lemma}\label{Lemma: continuity CBF-qp}
The S-CBF-QP problem in \eqref{eq:armin Son-CBF-qp} is feasible and has an unique solution for all $x\in\graphR_2 $. Additionally, the resulting control input $u^\ast (x)$ is Lipschitz continuous for all $x\in \graphR_2$.
\end{lemma}
For the proof we refer to \textbf{\textit{cite arXiv}} due to space limitations.

\begin{proof} The Lagrangian $\Lagr$
$(x,u,\lambda)$ 
for the S-CBF-QP problem is given by
\begin{equation}
\begin{aligned}\Lagr&=\|u-u_{son}(x)\|^2_{Q(x)}\\&-\lambda (L_fh(x)+\alpha_h(h(x))+L_gh(x)u), 
\end{aligned}
\end{equation}
where $\lambda$ is the Lagrange multiplier. For brevity, the argument $x$ is omitted in what follows. According to the Karush-Kuhn-Tucker (KKT) conditions \cite{boyd2004convex}, the solutions of the S-CBF-QP program are optimal and unique if the following conditions hold with respect to $u$:
\begin{equation}\label{eq: KKT conditions}
    \begin{aligned}
        \frac{\partial\Lagr}{\partial u}=(u-u_{son})^\top Q -L_gh\lambda &=0,\\ \lambda(L_fh+\alpha_h(h)+L_gh u)&=0.
    \end{aligned}
\end{equation}
Additionally, these conditions hold if and only if the cost is convex, i.e. $Q(x) \succ 0$ for all $x\in\graphC$ and the constraints are affine, \cite{boyd2004convex}. 
For $x\in\graphR_2$ the KKT conditions result in
\begin{align}\label{eq: KKT conditions R2}
    L_fh+\alpha_h(h)+L_ghu&=0,\\
    \lambda &\geq 0.\label{eq: KKT conditions R2 pt2}
\end{align} 
From \eqref{eq: KKT conditions} and \eqref{eq: KKT conditions R2 pt2} the following solution is obtained
\begin{equation}
\begin{aligned}\label{eq: KKT sol R2}
 \lambda&=-\left(\frac{L_fh+\alpha_h(h)}{\|L_gh^\top\|_2^2}L_gh^\top+u_{son}\right)\frac{QL_gh^\top}{\|L_gh^\top\|_2^2},\\
 u^\ast&=-\frac{L_fh+\alpha_h(h)}{\|L_gh^\top\|_2^2}L_gh^\top.
\end{aligned}
\end{equation} 
This solution holds for all $x\in \graphR_2 $. Additionally, the solution is well defined as $L_gh$ is bounded away from 0 in any bounded subset of $\graphR_2 $ as $h$ is a CBF for system \eqref{eq:2.1}. 
Since the KKT conditions are met and the unique solution for the S-CBF-QP problem is given by \eqref{eq: KKT sol R2}, the S-CBF-QP problem is feasible for all $x\in\graphR_2$.
 
Lastly, we need to prove Lipschitz continuity of $u^\ast(x)$ for all $x\in\graphR_2$. 
The coefficient matrix of the S-CBF-QP problem is given by
\begin{align}
    M(x)=
    \begin{bmatrix}
    -L_gh(x)\end{bmatrix},\quad\forall x\in \graphR_2. 
\end{align}
Following Theorem 3.1 of \cite{hager1979lipschitz}, in order for $u^\ast(x)$ to be Lipschitz continuous for all $x\in\graphR_2 $, the coefficient matrix $M(x)$ must be full rank for all $x\in\graphR_2$. Given that $L_gh(x)$ is bounded away from zero for all $x\in \graphR_2 $, the matrix $M(x)$ is full rank, ensuring that $u^\ast (x)$ is Lipschitz continuous for all $x\in\graphR_2 $ and thus completing the proof.
\end{proof}

\begin{remark}
    The S-CBF-QP problem is utilized as part of the hybrid S-CBF-QP control law in \eqref{Eq: Hybrid control}, and thus it is only employed for $x\in\graphR_2$ to compute a control input. For $x\in\graphR_1$, Sontag's formula is used instead. Therefore, it is not necessary to analyze Lipschitz continuity and the uniqueness of the solution of the S-CBF-QP problem for $x\in\graphR_1$.
\end{remark}

\begin{corollary}\label{Lemma: continuity hybrid control}
    Assuming that $\graphR_1$ is a $x_e$-proper set, the hybrid control law $\phi(x)$ in \eqref{Eq: Hybrid control} is continuous for all $x\in \graphC$.
\end{corollary}
The above result directly follows from: 
\begin{itemize}
    \item[\emph{(i)}] The switching condition: for all $x\in\partial\graphR_1\cup\partial\graphR_2$, i.e. when $L_fh(x)+L_gh(x)u_{son}(x)=-\alpha_h(h(x))$, $\phi(x)=u^\ast(x)=u_{son}(x)$;    
    \item [\emph{(ii)}] For any $x\in\graphR_1$,  $\phi(x)= u_{son}(x)$; 
    \item[\emph{(iii)}] The continuity of $u_{son}(x)$ and of $u^\ast(x)$, which were shown in Lemma \ref{lemma:Sontag} and Lemma \ref{Lemma: continuity CBF-qp}, respectively. 
\end{itemize}

\begin{figure}[H] 
    \centering
\includegraphics[width=0.5\textwidth]{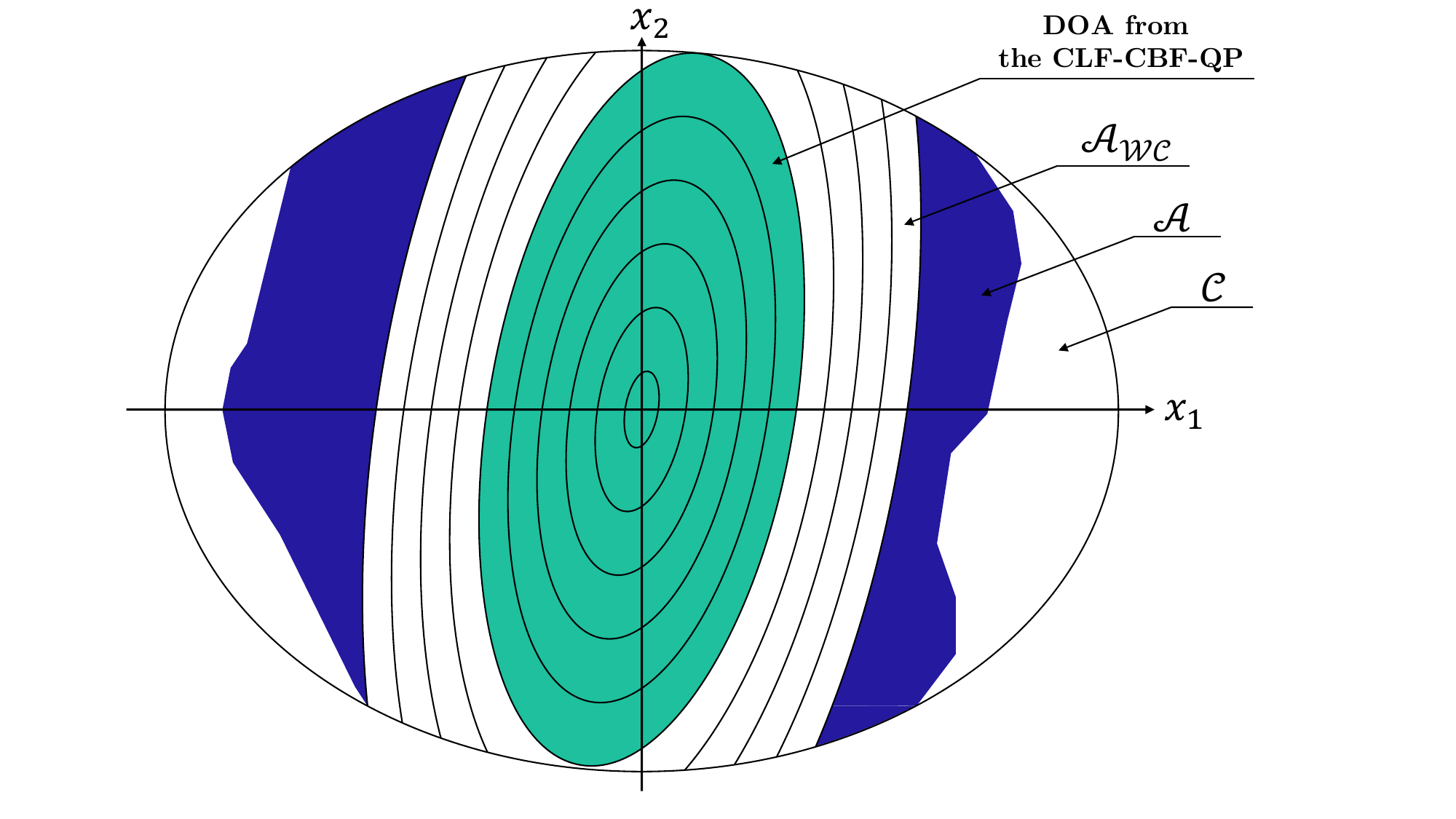}
    \caption{Graphical illustration of the sets $\graphW^\ast$, $\graphC$ and $\asset$.}
    \label{fig:sets}
\end{figure}

\begin{definition}[\hspace{-0.01cm}\cite{Controlsharing}]\label{Def: Control sharing property}
A CLF $W$ and a CBF $h$ for system \eqref{eq:2.1} have the control sharing property, if for some $x_e$-proper set $\graphA\subseteq\setRn$, there exists a control input, such that the following inequalities are simultaneously satisfied:
\begin{align}
        \nabla h(x)^\top (f(x)+ g(x)u)&\geq -\alpha_h(h(x))\quad \forall x\in\graphA,\label{eq:contsharcbf}\\
        \nabla W(x)^\top (f(x)+g(x)u)&<0\quad \forall x\in\graphA\setminus\{x_e\}.\label{eq:contsharclf} 
\end{align}
\end{definition}

Note that $\graphA\subseteq\graphC$.  A set $\graphW$ is defined as the sub-level set of the CLF $W$, i.e. $\graphW:=\{x\in\setRn \;|\; W(x)\leq c\}$ where $c>0$. Let $c^\ast$ be defined such that 
\begin{equation}\label{eq: wstar}
    \begin{aligned}
    c^\ast =\;\text{arg} \max_c \;\; &c\\
    \textrm{s.t.} \quad &c>0\\
    &\graphW\cap\graphC\subseteq\graphA,
    \end{aligned}
\end{equation}
let $\graphW^\ast:=\{x\in\setRn \;|\; W(x)\leq c^\ast\}$ and define $\asset:=\graphW^\ast\cap\graphC$. For a visual example of $\graphA$ and $\asset$, see Figure \ref{fig:sets}. 
The main result is stated next.
\begin{theorem}
Consider system \eqref{eq:2.1} with the equilibrium pair $(x_{e},u_{e})$, a CLF $W$ and a CBF $h$ with the associated safe set $\graphC$ for system \eqref{eq:2.1}. Assume $W$ and $h$ have the control sharing property for all $x\in\graphA$ as in Definition \ref{Def: Control sharing property} and $\graphA,\graphR_1,\asset\neq\emptyset$.
Then, the equilibrium $x_e$ of system \eqref{eq:2.1} in closed-loop with $u=\phi(x)$ as defined in \eqref{Eq: Hybrid control} is asymptotically stable for all $x\in\asset$. Additionally, the hybrid control law $\phi(x)$ renders the set $\asset$ controlled invariant for the closed-loop system \eqref{eq:2.1}-\eqref{Eq: Hybrid control}.
\end{theorem}
For the proof we refer to \textbf{\textit{cite arXiv}} due to space limitations.

\begin{proof}
$\phi(x)$ switches between two control laws based on a state switching condition, thus we consider two cases.

\textbf{Case I:} $x\in\graphR_1\cap\asset$\\
For $x\in\graphR_1\cap\asset$, we have that $\phi(x)=u_{son}(x)$ and based on the condition which defines the set $\graphR_1$, we have that $L_fh(x)+L_gh(x)u_{son}(x)\geq -\alpha_h(h(x))$ for all $x\in\graphR_1$. In Lemma \ref{lemma:Sontag}, we show that this results in 
\begin{align}\label{eq:SonLyap}
    \nabla W(x)^\top(f(x)+g(x)u_{son})<0, \, \forall x\in\graphR_1\cap\asset\setminus\{x_e\}.
\end{align}

\textbf{Case II:} $x\in\graphR_2\cap\asset$\\
For $x\in\graphR_2\cap\asset$ we have that $\phi(x)=u^\ast(x)$. First, we can express the cost function of the S-CBF-QP problem as follows:
\begin{align}\label{eq: rewritten optimization}
     \|u&-u_{son}\|^2_{Q(x)}\\
    &=\|\nabla W(x)^\top g(x)(u-u_{son})\|^2_2\\
    &=\|\nabla W(x)^\top (f(x)-f(x)+g(x)(u-u_{son}))\|^2_2\\  &=\|\nabla W(x)^\top \Dot{x}(x,u)-\nabla W(x)^\top \Dot{x}(x,u_{son})\|^2_2.\label{eq:min costfunc}
\end{align}
Given that $\graphR_2\cap\asset\subseteq\graphA$, and considering \eqref{eq:min costfunc}, we can use equations \eqref{eq:armin Son-CBF-qp}, \eqref{eq:contsharcbf},  \eqref{eq:contsharclf} and \eqref{eq:min costfunc}, to establish: 
\begin{align}
    \nabla W(x)^\top(f(x)+g(x)\phi(x))&<0\label{eq:clfasset}\\
    \nabla h(x)^\top (f(x)+g(x)\phi(x))&\geq -\alpha_h(h(x))\label{eq:cbfasset}
\end{align}
for all $x\in\graphR_2\cap\asset$. Since  $(\graphR_1\cap\asset)\cup (\graphR_2\cap\asset)=\asset$, then, from  \eqref{eq:SonLyap}, \eqref{eq:clfasset} and \eqref{eq:cbfasset} it follows that 
\begin{align}
\nabla h(x)^\top(f(x)+g(x)\phi(x))&\geq\label{eq:barAsset}\\
-\alpha_h(h&(x))\;\forall x \in \asset,\notag\\
    \nabla W(x)^\top(f(x)+g(x)\phi(x))&<0 \; \forall x \in \asset\setminus\{x_e\}.\label{eq:LyapAsset}
\end{align}

Next, let us prove that $\phi(x)$ renders the set $\asset$ invariant for system \eqref{eq:2.1}. By \eqref{eq:barAsset} and \eqref{eq:LyapAsset}, we have that, $f(x)+g(x)\phi(x)\in\graphT_\graphC(x)$ and $f(x)+g(x)\phi(x)\in\graphT_{\graphW}^\ast(x)$ for all $x\in\partial\asset$, i.e.  $f(x)+g(x)\phi(x)\in\graphT_{\graphC}(x)\cap\graphT_{\graphW^\ast}(x)$. Given that $\graphC$ and $\graphW^\ast$ are both convex sets, it follow that $f(x)+g(x)\phi(x)\in\graphT_{\graphC\cap{\graphW^\ast}}(x)=\graphT_{\asset}(x)$, following \cite{bazaraa1974cones}. According to Nagumo's theorem \cite{blanchini2008set}, if $f(x)+g(x)\phi(x)\in \tangcone$ for all $x\in\partial\asset$, the set $\asset$ is rendered invariant by $\phi(x)$ for system \eqref{eq:2.1}.

Hence, we have that $W$ is a CLF in $\asset$ and since $\asset$ is controlled invariant for the derived hybrid control law, the claim follows using a similar reasoning as in the proof of Lemma~\ref{lemma:Sontag}.
\end{proof}

To demonstrate the effectiveness of the developed hybrid control law \eqref{Eq: Hybrid control}, we follow the example shown in Section \ref{subsec: Example}. The same CLF $W$ and CBF $h$ are used. As can be seen in Figure \ref{fig1}, the states converge to the desired equilibrium while still maintaining the safety guarantees. Additionally, by tuning $\gamma$, the tuning parameter that dictates the convergence rate introduced in Lemma~\ref{lemma:Sontag}, the decrease rate can be tuned to allow for a more conservative input. Note that in Figure \ref{fig1}, the input makes a sharp increase around $2$ seconds. This abrupt change of $u$, corresponds to the moment the CBF constraints go from active to inactive and occurs for all mentioned CBF-QP and CLF-CBF-QP problems. I.e., the abrupt change does not originate from the switching in our hybrid S-CBF-QP control law.

\begin{figure}[H]
    \centering
      \includegraphics[width=1\linewidth]{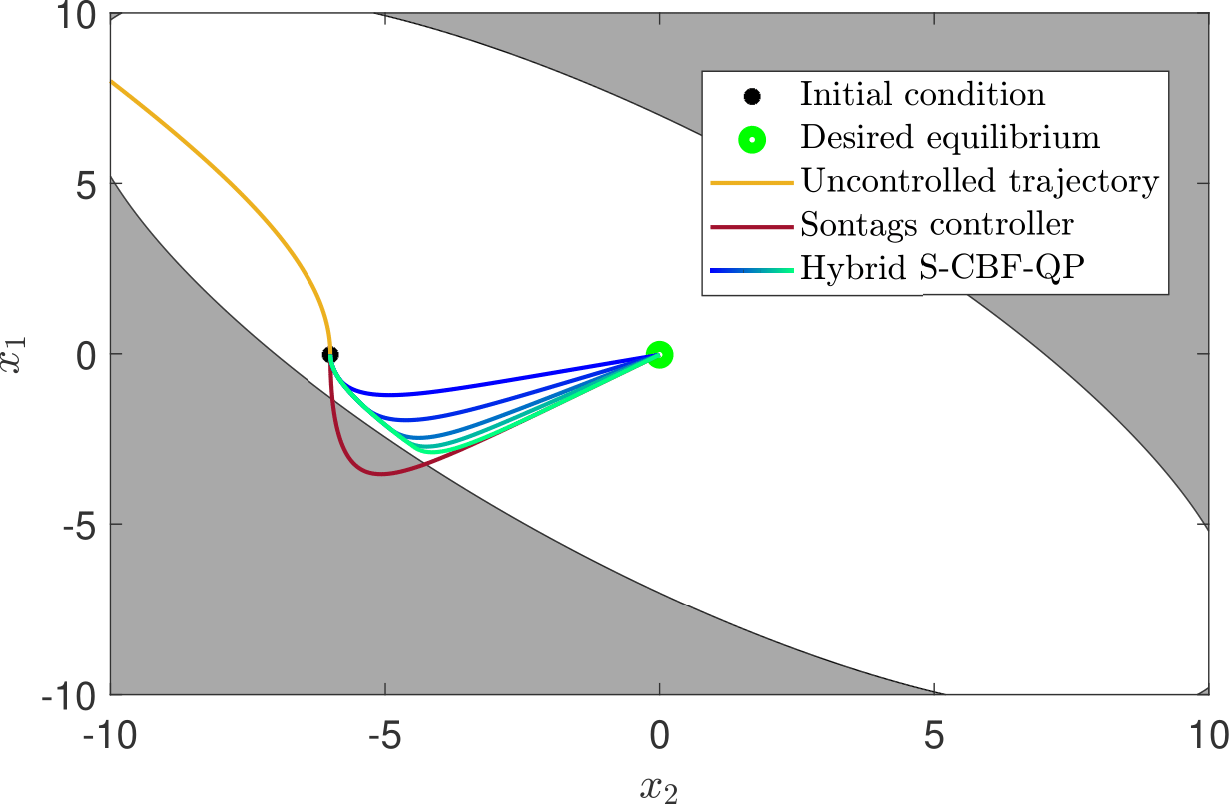}
\end{figure}
\vspace{-7mm}
\begin{figure}[H]
    \centering
    \includegraphics[width=1.02\linewidth]{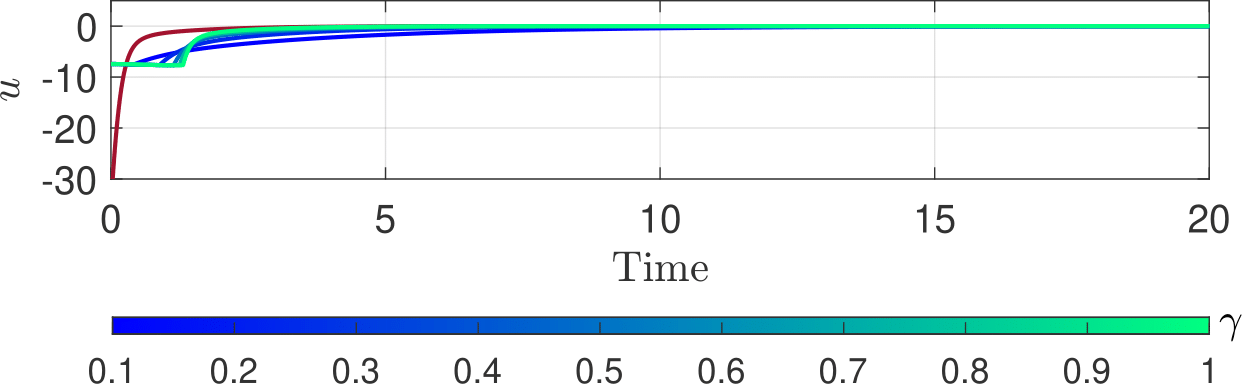}
 \caption{Example \eqref{eq:ex system} with the introduced hybrid S-CBF-QP control law for different $\gamma$ with the values of $\gamma$ defined in the color bar.}
 \label{fig1} 
\end{figure}

\section{Realistic stabilization of tumor dynamics}
In this section, we apply the developed S-CBF-QP to the stabilization of tumor dynamics using the effects of immunotherapy with the developed hybrid control law \eqref{Eq: Hybrid control}.

Consider the nonlinear tumor model as system \eqref{eq:2.1} with 
\begin{equation}\label{eq:2.4}
\begin{aligned}
    f(x)&=\begin{bmatrix}
        R_T x_1-\frac{R_T}{K_T}x_1^2-\frac{\alpha_{TN}R_T}{K_T}x_1x_2\\
        -\alpha_{NT}x_2x_1+\beta x_2x_3\\
        R_Rx_3-\frac{R_R}{K_R}x_3^2-\frac{\beta R_R}{K_R}x_2x_3
    \end{bmatrix},\\ g(x)&=\begin{bmatrix}-\frac{R_T}{K_T}x_1x_2\\0\\0
    \end{bmatrix}
\end{aligned}
\end{equation}
which  describes the regression and progression of normal and tumor cells as proposed in \cite{doban2015evolutionary}. 
The tumor cell population is denoted by $x_1$ and the resting and hunting immune cell populations are denoted by $x_2$ and $x_3$, respectively.

The parameters are chosen as follows, $\alpha_{NT}=0.5$, $\alpha_{TN}=0.9$, $\beta=0.9$, $K_R=K_T=10$ and $R_R=R_T=0.9$ as suggested in \cite{doban2015evolutionary}. 

Next, consider the desired equilibrium point $(x_e,u_e)=\left(\begin{pmatrix}6.4286& 7.1429& 3.5714\end{pmatrix}^\top,-0.4\right)$ representing tumor dormancy \cite{doban2015evolutionary}. By linearizing system \eqref{eq:2.4} around the desired equilibrium point, the technique of \cite{boyd1994linear} can be used to compute a quadratic CLF $W(x)=x^\top Px$ where 
\begin{equation*}
    P=\begin{bmatrix}
    0.3564 & -0.2472 & -0.4017\\
   -0.2472 &  0.4597 &  0.3699\\
   -0.4017 &  0.3699 &  4.6665\\
   \end{bmatrix}.
\end{equation*}
For this system we design and implement three CBFs which are given by $h_1(x)=-(x_1-5)^2+25$, $h_2(x)=1-e^{-x_2}$ and $h_3(x)=1-e^{-x_3}$. All three CBFs are chosen to ensure the positivity of the controlled state trajectories. Additionally, $h_1$ is defined to limit the amount of tumor cells in the body. All three CBFs are implemented as separate CBF constraints. 

\begin{figure}[H]
    \centering
\hspace*{-0.3cm}                                   
\includegraphics[width=0.5\textwidth]{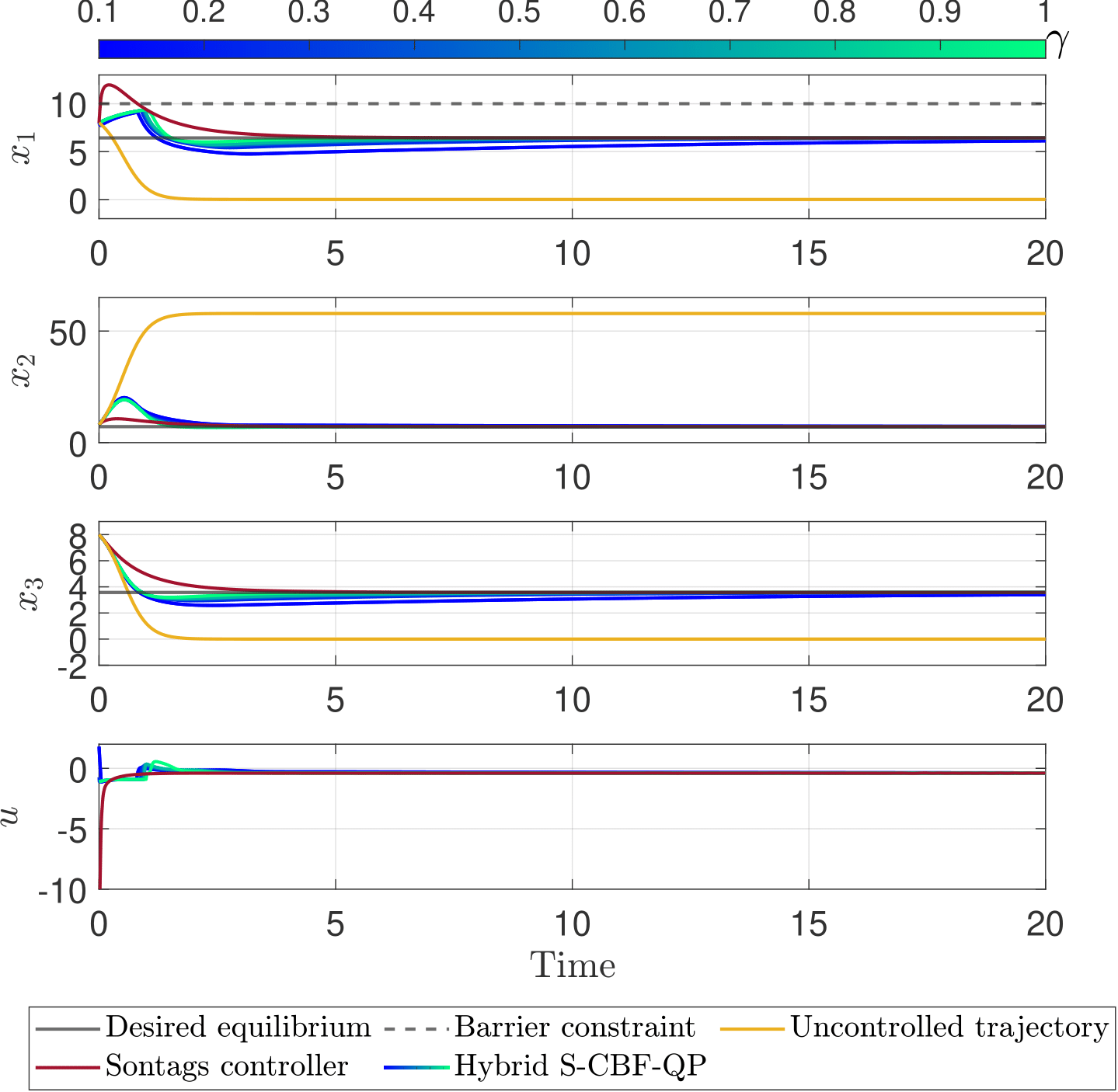}
    \caption{Stabilization of equilibrium representing tumor dormancy with the introduced hybrid S-CBF-QP control law for different $\gamma$ with the values of $\gamma$ defined in the color bar.}
    \label{fig:tumor}
\end{figure}

As can be seen in Figure \ref{fig:tumor}, the trajectories remain positive and the tumor cells within the bound defined by $h_1$, while converging towards the desired equilibrium. Multiple trajectories are shown for different values of $\gamma$. As can be seen, $\gamma$ has a direct effect on the rate with which the states converge towards the equilibrium. For invasive tumors, $\gamma$ can be tuned such that the states converge faster towards the equilibrium. However, for cases with non-invasive tumors, $\gamma$ can be tuned such that the treatment is less invasive, i.e. with less drug usage.

\section{Conclusions}
In this paper, we proposed an alternative CBF-QP problem formulation that uses a state-dependent weighting matrix to implicitly enforce the satisfaction of a CLF constraints, when feasible. Compared to previous CLF-CBF-QP problems, the derived formulation does not require an additional paramter $\delta$, which eliminates the need to tune the penalty of this parameter for every initial state. Based on the developed CBF-QP problem we derive a hybrid control law that switches locally to Sontag's formula and we provide an explicit domain of attraction for the resulting closed-loop system, which is less conservative than previous DOA estimates in the literature. The proposed hybrid control law was used to stabilize a desired equilibrium presenting a state of tumor dormancy. Via this example, we have shown that this method is computationally effective and also adjustable depending on the preference of treatment. 

\bibliographystyle{IEEEtran}
\bibliography{Bibliography}

\end{document}